\documentclass[11pt,letterpaper]{amsart}
\usepackage{graphicx}
\usepackage{amsmath, amssymb, amsthm, tikz, mathrsfs,verbatim,tensor,pstricks}
\usepackage[left=1in,top=1in,right=1in]{geometry}
\usepackage{setspace}
\usepackage{enumerate}

\usepackage{amsfonts}
\usepackage{latexsym}
\usepackage{graphics}

\usepackage{bbm}

\usepackage{etoolbox}
\patchcmd{\section}{\scshape}{\bfseries}{}{}
\makeatletter
\renewcommand{\@secnumfont}{\bfseries}
\makeatother

\newtheorem{theorem}{Theorem}
\newtheorem{lemma}{Lemma}[section]
\newtheorem{corollary}[theorem]{Corollary}

\theoremstyle{definition}

\theoremstyle{plain}

\newcommand{\beqlbl}{\begin{equation}}
\newcommand{\eeqlbl}{\end{equation}}

\newcommand{\sn}{S_n}

\newcommand{\R}{\ensuremath{\mathbb{R}} }

\newcommand{\cf}{\mathbf{1}}
\renewcommand{\P}{\mathbb{P}}

\newcommand{\cU}{\mathcal{U}}
\newcommand{\cT}{\mathcal{T}}

\usepackage{hyperref}
\usepackage{tikz}
\usepackage{xypic} 
\xyoption{curve}

\usetikzlibrary{decorations.markings}
\tikzstyle{vertex}=[circle, draw, inner sep=0pt, minimum size=6pt]
\tikzstyle{Vertex}=[circle, draw, inner sep=0pt, minimum size=14pt]
\tikzstyle{Vertexc}=[circle, draw, inner sep=0pt, minimum size=14pt, fill=blue!30]
\tikzstyle{vertexc}=[circle, draw, inner sep=0pt, minimum size=6pt, fill=red!40]
\tikzstyle{vertexcg}=[circle, draw, inner sep=0pt, minimum size=6pt, fill=green!70!black]

\newcommand{\beq}{\begin{equation*}}
\newcommand{\eeq}{\end{equation*}}
\newcommand{\ba}{\begin{align*}}
\newcommand{\ea}{\end{align*}}

\newcommand{\field}[1]{\mathbb{#1}}

\newcommand{\matbegin}[1]{\left (  \begin{array} {#1} }
\newcommand{\matend}{ \end{array} \right ) } 
\newcommand{\prob}{\field{P}}

\newcommand{\bt}{\mathbf{t}}

\newcommand{\fT}{\mathfrak{T}}

\newcommand{\avn}[1]{\textbf{A\!v}_n(#1)}

\begin{document}
\title{Fixed points of 321-avoiding permutations}

\author[Christopher Hoffman]{ \ Christopher~Hoffman$^{\star\ddagger}$}
\author[Douglas~Rizzolo]{ \ Douglas~Rizzolo$^{\bullet\dagger}$}
\author[Erik~Slivken]{ \ Erik~Slivken$^{\triangle \circ}$}

\thanks{\thinspace ${\hspace{-.45ex}}^\star$
Department of Mathematics,
University of Washington, 
Seattle, WA, 98195.
\hfill \\
\thinspace ${\hspace{-.45ex}}^\bullet$
Department of Mathematical Science,
University of Delaware,
Newark, DE, 19716
\hfill \\
\thinspace ${\hspace{-.45ex}}^\triangle$
Department of Mathematics,
University of California, Davis
Davis CA, 95616.
\hfill \\
${\hspace{-.45ex}}^\dagger$ 
Supported in part by NSF grant DMS-1204840
\hfil \\
${\hspace{-.45ex}}^\ddagger$ 
Supported in part by NSF grant DMS-1308645
\hfil \\
${\hspace{-.45ex}}^\circ$ 
Supported in part by NSF RTG grant 0838212
\hfil \\
Email:
\hskip.02cm
\texttt{hoffman@math.washington.edu; drizzolo@udel.edu; erikslivken@math.ucdavis.edu}
}

\date{\today}

\vskip1.3cm

\begin{abstract}
We describe the distribution of the number and location of the fixed points of permutations that avoid the pattern $321$ via a bijection with rooted plane trees on $n+1$ vertices.  Using the local limit theorem for Galton-Watson trees, we are able to give an explicit description of the limit of this distribution.           
\end{abstract}

\maketitle

\section{Introduction}

For a permutation $\tau\in \sn$ and $\sigma\in S_k$ with $k\leq n$, $\tau$ is said to contain an instance of the pattern $\sigma$ if there exists a subset of indices $i_1<i_2<\cdots i_k$ such that $\tau_{i_s}<\tau_{i_t}$ if and only if $\sigma_s < \sigma_t$ for all $1\leq s< t\leq k.$  The class of permutations that do not contain an instance of $\sigma$ are called $\sigma$-avoiding and $\avn{\sigma}$ denotes the set of $\sigma$-avoiding permutations in $\sn$.  

In recent years much work has been done in understanding the statistics of various classes of pattern-avoiding permutations. One statistic that has received particular interest is the fixed point statistic \cite{elizalde2004bijections, hrs2, miner2014shape, robertson2002refined}.  From a combinatorial perspective, much of the work concerning statistics of pattern-avoiding permutation has been in understanding how they are affected by bijections from one pattern-avoiding class to another.  Robertson, Saracino, and Zeilberger \cite{robertson2002refined} found a bijection between $\avn{321}$ and $\avn{132}$ that preserves the fixed point statistic.  Elizalde and Pak \cite{elizalde2004bijections} give another bijection between the same classes that preserves both fixed points and the number of exceedances.  From the beginning there as been interest in analyzing the typical values of this statistic.  For example, the expected number of fixed points of a uniformly random element of $\avn{\sigma}$ is computed in \cite{elizaldethesis} for several cases when $\sigma$ has length $3$.  
    
A number of recent papers \cite{bassino2016brownian, hrs2, hrs1, janson2014patterns, madras2015structure, miner2014shape} have continued this study of pattern-avoiding permutations as random objects.  In the context of random permutations, the natural object to consider is the empirical distribution of fixed points.  In the authors' previous work in \cite{hrs2} they describe the limiting distribution of the number and location of fixed points for both $\avn{123}$ and $\avn{231}$ using the global connection with Brownian excursion they established in \cite{hrs1}.  In this paper we show that the empirical distribution of fixed points for random permutations in $\avn{321}$ is related to another classical probabilistic object, specifically a size-biased Galton-Watson tree.  

In \cite{hrs1} the authors also showed that if $\tau_n \in \avn{321}$ is uniformly random then
\[ \left( \frac{1}{\sqrt{2n}} \left|\tau_n(\lfloor nt\rfloor) - \lfloor nt\rfloor\right| \right)_{0\leq t\leq 1}\]
converges in distribution to standard Brownian excursion.  This describes the structure of the bulk of the permutation, however because of the scaling one cannot extract precise information about $\tau_n$ near $1$ and $n$.  Moreover, the only information about the fixed points we can obtain from this picture is that they must occur near the beginning or end of the permutation since Brownian excursion is strictly positive except at the end points.  To understand the behavior of a permutation $\tau_n$ near the beginning and the end, we utilize a bijection with rooted plane trees on $n+1$ vertices.  Rather than take a scaling limit, which would lead back to Brownian excursion, we take a local limit that results in a particular size-biased Galton-Watson tree.  The behavior of the permutation $\tau_n$ near $1$ and $n$ can then be described using the limiting random tree and we use this to give an explicit description of the limiting empirical measure of fixed points.  To the best of our knowledge, this is the first use of local limits of random trees in the study of pattern-avoiding permutations and gives another illustration of the value of classical probabilistic limit theorems in the study of random pattern-avoiding permutations.

The bijection we use is a reinterpretation of a bijection found in \cite{elizalde2004bijections}.  The bijection in \cite{elizalde2004bijections} is from $\avn{321}$ to the set of  Dyck paths of length $2n$ .  Under this bijection the fixed points of $\tau$ are sent to peaks of height 1.  We compose this bijection with the standard bijection from Dyck paths to rooted ordered trees the treats the Dyck path as the contour process of the tree (creating a new edge with upsteps and moving toward the root with a downstep).  Under this bijection the peaks of height 1 are sent to leaves connected to the root.  In section \ref{treebijection} we make this bijection explicit.

In Section \ref{locallimit} we describe the appropriate limiting object of a sequence of rooted plane trees.  Section \ref{fixedpoints} translates information about this limiting object to information about the number and location of fixed points in $\avn{321}.$  We can then simplify the description to obtain a clean statement that does not need the formalism of size-biased Galton-Watson trees. 

We state our main result, which gives the convergence of the joint distribution of the number and location of the fixed points at the beginning and end of a permutation chosen uniformly from $\avn{321}$.

\begin{theorem} \label{fp_distro}
Let $((X_i,Y_i))_{i\geq 1}$ be a i.i.d sequence such for each $i$, $X_i$ and $Y_i$ are independent with
\[ \P(X_i=k) = \P(Y_i=k) =  \frac{1}{k2^{2k-1}} { 2k-2 \choose k-1} .\]
Further, let $N$ and $M$ be independent Geometric$(1/2)$ random variables that are also jointly independent of the sequence $((X_i,Y_i))_{i\geq 1}$.  If $\tau_n$ is a uniformly random element of $\avn{321}$ then
\begin{multline*} \left( \sum_{i=1}^{\lfloor n/2\rfloor} \delta_i \cf_{\{\tau_n(i)=i\}} , \sum_{i=\lfloor n/2\rfloor+1}^n \delta_{n+1-i} \cf_{\{\tau_n(i)=i\}} \right)
\overset{d}{\longrightarrow} \left( \sum_{k=1}^{N} \delta_{\sum_{j=1}^k X_j} \cf_{\{X_k=1\}} ,\sum_{k=1}^{M} \delta_{\sum_{j=1}^k Y_j} \cf_{\{Y_k=1\}} \right)
,\end{multline*}
the convergence being convergence in distribution on $\mathcal{M}_F(\{1,2,3,\dots\})^2$ where $\mathcal{M}_F(\{1,2,3,\dots\})$ is the set of finite measures on $\{1,2,3,\dots\}$ equipped with the topology of weak convergence.
\end{theorem}

Recall that the topology of weak convergence on $\mathcal{M}_F(\{1,2,3,\dots\})$ is the topology generated by the set maps $\{ \mu \mapsto \smallint f d\mu \ | \ f:\{1,2,3,\dots\}\to \R \textrm{ is bounded}\}$. 

For comparison, the strongest prior result on the location of fixed points for random permutations in $\avn{321}$ is the following result from \cite{miner2014shape}.

\begin{theorem}{{\cite[Theorem 6.1]{miner2014shape}}}
Let $\tau_n$ be a uniformly random element of $\avn{321}$.  For fixed $0<\epsilon \leq 1/2$ we have
\[ \lim_{n\to \infty} \P\left( \tau_n(i) = i \textrm{ for some } i\in [\epsilon n, (1-\epsilon)n] \right) =0 .\]
\end{theorem} 

From our Theorem \ref{fp_distro}, we immediately obtain the following stronger result.

\begin{corollary}
Let $\tau_n$ be a uniformly random element of $\avn{321}$ and let $0< a_n \leq b_n$ be sequences of integers.  Then
\[ \lim_{n\to \infty} \P\left( \tau_n(i) = i \textrm{ for some } i\in [a_n, b_n] \right) =0\]
if and only if $a_n \to \infty$ and $n-b_n \to \infty$.
\end{corollary}

If we only wish to know the limiting distribution of the number of fixed points at the beginning and the end of the permutation we have a much simpler solution:

\begin{corollary} \label{321count}
	Let $A$ and $B$ be independent Geometric(2/3) random variables.  If $\tau_n$ is a uniformly random element of $\avn{321}$ then 
	\beq \left( \sum_{i=1}^{\lfloor n/2\rfloor} \cf_{\{\tau_n(i)=i\}} , \sum_{i=\lfloor n/2\rfloor+1}^n  \cf_{\{\tau_n(i)=i\}} \right)
\overset{d}{\longrightarrow} \left( A, B \right).
\eeq

\end{corollary}

The asymptotic distribution of the total number of fixed points can be also be found using the continuity of addition. 

\begin{corollary}\label{321total}

Let $Z$ be distributed as $\P(Z=k) = \frac{4}{9}(k+1)\left(\frac{1}{3}\right)^k$ and let $A$ and $B$ be independent Geometric(2/3) random variables.  If $\tau_n$  is a uniform random element of $\avn{321}$ then

\beq 
\sum_{i=1}^n \cf_{\{\tau_n(i) = i\}} \overset{d}\longrightarrow A+B \overset{d}{=} Z.
\eeq
	
\end{corollary}
\noindent We remark that $Z$ follows a NegativeBinomial$(2,1/3)$ distribution.  Although Corollary \ref{321total} seems to be new, and we derived from our explicit description of the limiting empirical distribution of fixed points, it can also be found by analyzing the bivariate generating function for fixed points given by \cite[Theorem 2.9]{elizaldethesis}.

Using one of the many bijections \cite{krattenthaler2001permutations,elizalde2004bijections,robertson2002refined} from $\avn{321}$ to $\avn{132}$ or $\avn{213}$ that preserves the number of fixed points we have the following.

\begin{corollary} \label{132count}
Let $Z$ be distributed like the sum of two independent Geometric$(2/3)$ random variables, i.e. $\prob(Z=k) = \frac{4}{9}(k+1) \left(\frac{1}{3}\right)^k$.  If $\pi_n$ is a uniformly random element in $\avn{132}$ or $\avn{213}$ then
$$\left(\sum_{i=1}^n \cf_{\{\pi_n(i) = i\}}\right) \overset{d}{\longrightarrow} Z.$$
\end{corollary}
\noindent Unlike the results for $\avn{321}$, these bijections do not easily give informations about the location of the fixed points for random permutations in $\avn{132}$ or $\avn{213}$.

\section{A bijection between rooted plane trees and $\avn{321}$}\label{treebijection}

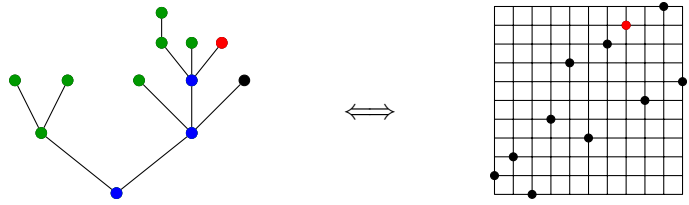
\begin{figure}

\begin{tikzpicture}[
every fit/.style={ellipse,draw,inner sep=2pt},
grow'=up,
interior/.style={draw, fill, circle, minimum size=4pt, inner sep=0pt, }, 
leaf/.style={draw, fill, circle, minimum size=4pt, inner sep =1pt},
subtree/.style={draw, circle, inner sep=1pt},level 1/.append style={sibling distance=20mm, level distance= 8mm},
level 2/.append style={sibling distance=7mm, level distance =7mm},
level 3/.append style={sibling distance=4mm, level distance = 5mm},
level 4/.append style={sibling distance=2mm, level distance = 4mm},
]
\node (root) [interior] {} 

 child{ node (7)[interior] {} 
  child{node (0) [leaf] {}
  }
  child{node (1) [leaf] {}
  }
 }
 child{ node (8)[interior] {} 
  child{node (2) [leaf] {}
  }
  child{ node (9)[interior] {} 
   child{ node (10)[interior] {} 
    child{node (3) [leaf] {}
    }
   }
   child{node (4) [leaf] {}
   }
   child{node (5) [leaf] {}
   }
  }
  child{node (6) [leaf] {}
  }
 }
;

	\node (h5) [interior, color=red,] at (5) {};

	\node (hr) [interior,color = green!60!black] at (root) {};
	\node (h1) [leaf, color=green!60!black] at (1) {};
	\node (h2) [leaf, color=green!60!black] at (2) {};
	\node (h3) [leaf, color=green!60!black] at (3) {};
	\node (h4) [leaf, color=green!60!black] at (4) {};
	\node (h7) [interior, color=green!60!black] at (7) {};
	\node (h8) [interior, color=green!60!black] at (8) {};
	\node (h9) [interior, color=green!60!black] at (9) {};
	\node (h10) [interior, color=green!60!black] at (10) {};
	\node (h11) [leaf, color=green!60!black] at (0) {};

	\node (br) [interior, color = blue] at (root) {};
	\node (b8) [interior, color = blue] at (8) {};
	\node (b9) [interior, color = blue] at (9) {};
	
\end{tikzpicture}
\hspace{.75cm}
\begin{tikzpicture}
\node (base) at (0,0) {};
\node (iff) at (0,1) {$\iff$};
\end{tikzpicture}	
\hspace{.75cm}
\begin{tikzpicture}[scale =.25,
interior/.style={draw, fill, circle, minimum size=3pt, inner sep=0pt, },]
\draw (1,1) grid (11,11);
	\node (l1) [interior] at (1,2) {};
	\node (l2) [interior] at (2,3) {};
	\node (l3) [interior] at (4,5) {};
	\node (l4) [interior] at (5,8) {};
	\node (l5) [interior] at (7,9) {};
	\node (l6) [interior] at (8,10) {};
	\node (l7) [interior] at (10,11) {};
	\node (i1) [interior] at (3,1) {};
	\node (i2) [interior] at (6,4) {};
	\node (i3) [interior] at (9,6) {};
	\node (i4) [interior] at (11,7) {};
	\node (r8) [interior, color = red] at (l6) {};
	
\end{tikzpicture}

\caption{ A rooted tree $\mathbf{t}$ and a corresponding permutation $\tau_n.$}
\label{fig bijection}
\end{figure}

Let $\fT_{n+1}$ denote the set of rooted plane trees with $n+1$.  We will describe a bijection with rooted plane trees and $\avn{321}$ that maps the leaves of the tree to the left-to-right maxima of the corresponding permutation, see Figure \ref{fig bijection}.  This bijection follows from a bijection with $\avn{321}$ and Dyck paths of length $2n$ found in \cite{elizalde2004bijections}.

Suppose $\bt$ has $k$ leaves and let $\ell(\bt) = (l_1,l_2,\cdots,l_k)$ be the list of leaves in $\bt$, listed in order of appearance on the depth-first walk of $\bt$.  There are two quantities associated with the leaves of the tree needed to construct the bijection.  Given the ordering on the vertices of $\bt$ given by traversing the exterior of the tree from left to right, starting from the root, let $s_i$ equal the number of vertices in $\bt$ less than $l_i$.  Let $p_i$ denote the number of vertices less than $l_i$ in the path from the root to $l_i$.  

Construct two increasing sequences of size $k$, $A = \{s_i\}_{i=1}^k$ and $B = \{ s_i - p_i + 1\}_{i=1}^k.$  We define $\tau_\bt$ pointwise on $B$
$$\tau_\bt( s_i-p_i + 1 ) = s_i$$ for $1\leq i \leq k$.  We then extend $\tau_\bt$ to $[n]\backslash B$ by assigning values of $[n]\backslash A$ in increasing order.  

The resulting function $\tau_\bt$ is a permutation in $\avn{321}$.  The set of values for $s_i$ and $p_i$ uniquely determine a rooted plane tree, so we may recover the corresponding rooted tree from the left-to-right maxima of a permutation in $\tau_\bt$.  Let $M=\{m_i\}_{i=1}^k$ denote the set of indices of the left-to-right maxima of $\tau_\bt$. Then
$$s_i = \tau_\bt(m_i)$$
and 
$$p_i = \tau_\bt(m_i) - m_i + 1.$$

Finally we note that fixed points in $\tau_\bt$ correspond to leaves $l_i$ of $\bt$ such that $p_i = 1$.

\section{The Local Limit Theorem}\label{locallimit}
In this section we introduce the appropriate limit for a sequence of trees $(\bt_n)_{n\geq 1}$ such that $\bt_n\in\fT_n$.  For this we use Neveu's formalism for rooted ordered trees \cite{N86}.  This will give us a consistent way to refer to thte vertices of different trees.  We let
\[ \cU = \bigcup_{n\geq 0} \{1,2,\dots\}^n,\]
where $\{1,2,\dots\}^0 = \{\emptyset\}$.  For $u\in\cU$, we let $|u|$ be the length of $u$.  That is, $|u|$ is the unique integer such that $u\in \{1,2,\dots\}^{|u|}$.  We treat the elements of $\cU$ as lists, and for $u,v\in \cU$, we denote the concatenation of $u$ and $v$ by $uv$.  The empty list is the identity for this operation, that is, if $u=\emptyset$ then $uv=vu=v$.  The set of ancestors of $u$ is defined to be
\[ A_u = \{ v \in \cU  : \exists w\in \cU \textrm{ such that } u=vw\}.\]
We can consider $\cU$ as the vertex set of a rooted ordered tree $U_\infty$, called the Ulam-Harris tree, where we take the lexicographic order on $\cU$ as the order, $\emptyset$ as the root, and for every $u\in \cU$ and $i\in \{1,2,3,\dots\}$ there is an edge from $u$ to $ui$.  By a rooted, ordered tree we will mean a (possibly infinite) subtree $\bt$ of $U_\infty$ such that
\begin{enumerate}
\item $\emptyset\in \bt$.
\item If $u\in\bt$ then $A_u\subseteq \bt$.
\item For every $u\in  \bt$, if $ui\in \bt$ then $uj\in \bt$ for all $j \leq i$.
\end{enumerate}

A tree $\bt\subseteq U_\infty$ is considered as rooted and ordered by taking $\emptyset$ to be the root and using the order on $\bt$ induced by the lexicographic order on $\cU$.  If $u\in \bt$, the vertices in $A_u$ are the vertices in $\bt$ on the path from the root to $u$.  For a tree $\bt$ and $u\in \bt$, define $d_\bt(u)=0 \vee \sup\{ i : ui\in\bt\}$.  We also define $d_\bt(u)=-1$ if $u\notin \bt$.  A tree $\bt$ is called \textit{locally finite} if $d_\bt(u)<\infty$ for all $u\in\cU$ and define
\[ \fT = \{ \bt\subseteq U_\infty : \bt \textrm{ is locally finite}\}.\]
We topologize $\fT$  with the convergence $\bt_n\rightarrow \bt$ if $d_{\bt_n}(u) \rightarrow d_{\bt}(u)$ for all $u\in \cU$.  Since $\cU$ is countable this is convergence can be metrized by, for example, the complete metric
\[ D(\bt,\mathbf{s}) = \sum_{i=1}^\infty \frac{ |d_{\bt}(f(i)) - d_{\mathbf{s}}(f(i))| \wedge 1}{2^i},\]
where $f: \{1,2,\dots, \}\to \cU$ is a bijection.

Let $\xi$ be a probability distribution on $\{0,1,2,\dots\}$ with mean $1$, ie. $\sum_{i} i\xi(i) =1$.  A $\fT$-valued random variable is called a Galton-Watson tree with offspring distribution $\xi$ if
\[ \P(T = \bt) = \prod_{u\in \bt} \xi(d_\bt(u))\]
for all $\bt \in \fT$.  For $\bt\in\fT$, let $\bt^{[k]}$ be the subtree of $\bt$ comprised of vertices of height at most $k$.  Let $T$ be a $\xi$-Galton-Watson tree and define the probability measure $\nu_\xi$ on $\fT$ to be the unique measure on infinite trees such that for every $k$ and $\bt_0\in \fT$ with height $k$, 
\[ \nu_\xi(\{ \bt \in \fT : \bt^{[k]} = \bt_0\}) = (\#_k\bt_0) \P( T^{[k]} = \bt_0),\]
where $\#_k\bt$ is the number of vertices in $\bt$ with height $k$.  See e.g. \cite[Lemma 1.14]{Kesten86} for the existence and uniqueness of this measure.    Let $\tilde T$ have distribution $\nu_\xi$.  $\tilde T$ is called a size-biased $\xi$-Galton-Watson tree.  The size biased distribution of $\xi$ is the distribution
\[ \tilde \xi(n) = n\xi(n).\]
Note that this is a probability distribution because $\xi$ has mean $1$.  Observe that, directly from the definition of $\nu_\xi$, we have
\[ \P(\#_1 \tilde T = k) = \tilde \xi(k),\]
which justifies calling $\tilde T$ the size-biased tree.  For trees $\bt_1,\dots, \bt_k$, let $<\bt_1,\dots,\bt_k>$ be the rooted ordered tree obtained by attaching the roots of $\bt_1,\dots, \bt_k$ to a new vertex and calling that vertex the root.  Formally, for $\bt\in \fT$ and $i \in \{1,2,\dots\}$, we let $i\bt = \{ iv \in \cU : v\in \bt\}$ and define $<\bt_1,\dots,\bt_k>$ as the rooted, ordered tree with vertex set
\[ \{\emptyset\} \cup \bigcup_{i=1}^k i\bt_i.\]

For our purposes, $\tilde T$ has two important properties, which we summarize in the next two lemmas.
\begin{lemma}[{\cite[Lemma 2.2]{Kesten86}} ]\label{lemma sbgwdecomp}
Let $T_1,T_2,\dots$ be an i.i.d sequence of $\xi$-Galton-Watson trees, let $\tilde T$ be an independent size-biased $\xi$-Galton-Watson trees and let $(\tilde X, N)$ be independent of $\tilde T, T_1,T_2,\dots$ and distributed such that $\tilde X$ has distribution $\tilde \xi$ and conditionally give $\tilde X=k$, $N$ is uniformly distributed on $\{1,2,\dots, k\}$.  Then $\tilde T$ is distributed like
\[ \tilde T \ =_d \ <T_1,\dots, T_{N-1}, \tilde T, T_{N+1},\dots, T_{\tilde X}>. \]
\end{lemma}  

See also the discussions in Section 2 of \cite{LPP95} and Section 5 of \cite{J12} for a constructions of $\tilde T$ from which the lemma follows immediately.  For $\bt \in \fT$, let $\#\bt$ be the number of vertices of $\bt$.

\begin{lemma}\label{lemma local limit}
Let $T$ be a $\xi$-Galton-Watson tree and for $n$ such that $\P(\# T=n)>0$, let $T_n$ be distributed like $T$ conditioned on $\P(\# T=n)$.  Then
\[ T_n \overset{d}{\longrightarrow} \tilde T.\] 
That is, for every tree $\bt\in \fT$ and $k\in \{1,2,3,\dots\}$ we have
\[ \lim_{n\to\infty} \P\left(T_n^{[k]} = \bt\right)=\lim_{n\to\infty} \P\left(T^{[k]} = \bt \middle| \# T = n\right) = \P\left(\tilde T^{[k]} = \bt\right),\]
where the limit is understood to be along $n$ such that $\P(\# T=n)>0$.
\end{lemma}

This lemma has a long history.  The case when $\xi$ has finite variance, which is the only case we will need, was proven implicitly in \cite{K75} and first stated explicitly in \cite{AP98}.  A version for conditioning on $\#T\geq n$ rather than $\#T = n$ was given in \cite{Kesten86}.  See \cite[Theorem 7.1]{J12} and \cite{AD14} for modern approaches to the general result.

\section{Application to fixed points}\label{fixedpoints}
The connection between local limit theorems for Galton-Watson trees and fixed points of $321$-avoiding permutations comes from the following result.  For $n\in \{0,1,2,\dots \}$ define $\xi(n) = 2^{-n-1}$, i.e., $\xi$ has the Geometric$(1/2)$ distribution.  If $T$ is a $\xi$-Galton-Watson tree and $T_n$ is distributed like $T$ conditioned on $\# T=n$, then $T_n$ is a uniformly random rooted ordered tree with $n$ vertices.  Thus, $\tau_n:=\tau_{T_{n+1}}$ is a uniformly random element of $\avn{321}$.

From the local limit theorem for Galton-Watson trees, we obtain the following result.

\begin{theorem}\label{thm main tree}
Let $\cT_1,\cT_2,\dots$ be an i.i.d. sequence of Geometric$(1/2)$-Galton-Watson trees, let $(\tilde X, N)$ be independent of $\cT_1,\cT_2,\dots$ and distribution such that $\tilde X$ has a size-biased Geometric$(1/2)$ distribution and conditionally give $\tilde X=k$, $N$ is uniformly distributed on $\{1,2,\dots, k\}$.  If $\tau_n$ is a uniformly random element of $\avn{321}$ then
\begin{multline*} \left( \sum_{i=1}^{\lfloor n/2\rfloor} \delta_i \cf_{\{\tau_n(i)=i\}} , \sum_{i=\lfloor n/2\rfloor+1}^n \delta_{n+1-i} \cf_{\{\tau_n(i)=i\}} \right)\\ \overset{d}{\longrightarrow} \left( \sum_{i=1}^{N-1} \delta_{\sum_{j=1}^i \#\cT_j} \cf_{\{\#\cT_i=1\}}, \sum_{i=N+1}^{\tilde X} \delta_{\sum_{j=i}^{\tilde X} \#\cT_j} \cf_{\{\#\cT_{i}=1\}}\right),
\end{multline*}
the convergence being convergence in distribution on $\mathcal{M}_F(\{1,2,3,\dots\})^2$ where $\mathcal{M}_F(\{1,2,3,\dots\})$ is the set of finite measures on $\{1,2,3,\dots\}$ equipped with the topology of weak convergence.
\end{theorem}

\begin{proof}
Let $T_{n+1}$ be distributed like $\cT_1$ conditioned on $\#\cT_1=n+1$.  Since $T_n \rightarrow_d \tilde T$ by Lemma \ref{lemma local limit} and $\fT$ admits a complete metric, we may use the Skorokhod Representation Theorem (see e.g. \cite[Theorem 4.30]{kallenberg2002foundations}) to construct a version of the sequence $(T_n)_{n\geq 0}$ and $\tilde T$ such that $T_n \rightarrow \tilde T$ almost surely.  For $\bt\in \fT$ and $1\leq i\leq d_\bt(\emptyset)$, define $F^\bt_i \in \fT$ to be the fringe subtree of $\bt$ above $i$.  That is, $F^\bt_i \in \fT$ is the rooted ordered tree with vertex set $\{ u\in \cU : iu\in \bt\}$.  In particular, note that $\bt = < F^\bt_1,\dots, F^\bt_{d_\bt(\emptyset)}>$.  

Define
\[ N_n = \inf\left\{ k: \sum_{i=1}^k \#F^{T_n}_k > \lfloor n/2\rfloor \right\} \quad \textrm{and} \quad N = \inf\left\{ k: \sum_{i=1}^k \#F^{\tilde T}_k =\infty \right\}.\]
By Lemma \ref{lemma sbgwdecomp}, $N$ is the unique index such that $ \#F^{\tilde T}_N=\infty$.  Letting $\tau_n = \tau_{T_{n+1}}$, we observe that 
\begin{multline*} \left( \sum_{i=1}^{\lfloor n/2\rfloor} \delta_i \cf_{\{\tau_n(i)=i\}} , \sum_{i=\lfloor n/2\rfloor+1}^n \delta_{n+1-i} \cf_{\{\tau_n(i)=i\}} \right)\\ = \left( \sum_{i=1}^{N_{n+1}-1} \delta_{\sum_{j=1}^i \#F^{T_{n+1}}_j} \cf_{\left\{\#F^{T_{n+1}}_i=1\right\}}, \sum_{i=N_{n+1}}^{d_{T_{n+1}}(\emptyset)} \delta_{\sum_{j=i}^{d_{T_{n+1}}(\emptyset)}\#F^{T_{n+1}}_j} \cf_{\left\{\#F^{T_{n+1}}_i=1\right\}}\right),
\end{multline*}
Since $T_n \rightarrow T$ almost surely there exists $M=M(\omega)$ such that for all $n\geq M(\omega)$, $d_{T_{n}}(\emptyset) = d_{\tilde T}(\emptyset)$, $F^{T_n}_j = F^{\tilde T}_j $ for all $1\leq j \leq d_{\tilde T}(\emptyset)$ such that $\#  F^{\tilde T}_j<\infty$, $N_n=N$, and $ \# F^{T_n}_{N_n} \geq 3n/4$.  Therefore, for $n \geq M(\omega)$ 
\begin{multline*} \left( \sum_{i=1}^{N_{n+1}-1} \delta_{\sum_{j=1}^i \#F^{T_{n+1}}_j} \cf_{\left\{\#F^{T_{n+1}}_i=1\right\}}, \sum_{i=N_{n+1}}^{d_{T_{n+1}}(\emptyset)} \delta_{\sum_{j=i}^{d_{T_{n+1}}(\emptyset)}\#F^{T_{n+1}}_j} \cf_{\left\{\#F^{T_{n+1}}_i=1\right\}}\right) \\
 = \left( \sum_{i=1}^{N-1} \delta_{\sum_{j=1}^i \#F^{\tilde T}_j} \cf_{\{\#F^{\tilde T}_i=1\}}, \sum_{i=N+1}^{d_{\tilde T}(\emptyset)} \delta_{\sum_{j=i}^{d_{\tilde T}(\emptyset)}\#F^{\tilde T}_j} \cf_{\{\#F^{\tilde T}_i=1\}}\right).
 \end{multline*}
Noting that
\begin{multline*}
\left( \sum_{i=1}^{N-1} \delta_{\sum_{j=1}^i \#F^{\tilde T}_j} \cf_{\{\#F^{\tilde T}_i=1\}}, \sum_{i=N+1}^{d_{\tilde T}(\emptyset)} \delta_{\sum_{j=i}^{d_{\tilde T}(\emptyset)}\#F^{\tilde T}_j} \cf_{\{\#F^{\tilde T}_i=1\}}\right) \\
=_d  \left( \sum_{i=1}^{N-1} \delta_{\sum_{j=1}^i \#\cT_j} \cf_{\{\#\cT_i=1\}}, \sum_{i=N+1}^{\tilde X} \delta_{\sum_{j=i}^{\tilde X} \#\cT_j} \cf_{\{\#\cT_{i}=1\}}\right),
\end{multline*}
completes the proof. 
\end{proof}

Using the special properties of the geometric distribution, we are able to simplify the statement of the theorem and, indeed, remove the trees from the picture entirely.  

\begin{proof}[Proof of Theorem \ref{fp_distro}]

Since $\P(\tilde X= j ) = j (1/2)^{j+1}$ we compute that for $k\geq 0$
\[ \P(N = k+1) = \sum_{j=k+1}^\infty \frac{1}{j}\P(\tilde X=j) = \sum_{j=k+1}^\infty \frac{1}{2^{j+1}} = \frac{1}{2^{k+1}}.\]
Thus $N-1$ has a Geometric$(1/2)$ distribution.  Similarly, for $k\geq 0$ we have
\[ \P(\tilde X-N= k) = \sum_{j=k+1}^\infty  \P(\tilde X-N= k, \tilde X=j) = \sum_{j=k+1}^\infty \frac{1}{j}\P(\tilde X=j)  = \frac{1}{2^{k+1}},\]
which shows that $\tilde X-N$ has a Geometric$(1/2)$ distribution.  Moreover for $k,j\geq 0$, 
\[ \begin{split} \P(N-1=k, \tilde X-N = j) = \P(N=k+1,\tilde X = j+k+1) & = \frac{1}{j+k+1} \P(\tilde X = j+k+1) \\
& = \frac{1}{j+k+1} (j+k+1)\frac{1}{2^{j+k+2}} \\
& =\P(N-1=k) \P(\tilde X-N= j)
.\end{split}\]
Thus $N-1$ and $\tilde X-N$ are independent Geometric$(1/2)$ random variables.  Furthermore, for $\bt\in\fT$, we have that
\[ \P(T_1=\bt) = \prod_{v\in\bt} \frac{1}{2^{d_\bt(v)+1}} = \frac{1}{2^{2\#\bt -1}},\] 
from which it follows that
\[ \P(\# T_1 = k) = \frac{1}{2^{2k-1}} C_{k-1} =  \frac{1}{k2^{2k-1}} { 2k-2 \choose k-1} \]
where $C_{k-1} = \frac{1}{k} { 2k-2 \choose k-1}$ is the $(k-1)$'st Catalan number.  
\end{proof}

\begin{proof}[Proof of Corollary \ref{321count}]

Let $Y_1 = \sum_{k=1}^N1_{\{X_k = 1\}}$ and $Y_2 = \sum_{k=1}^M\cf_{\{Y_k = 1\}}.$  By Theorem \ref{fp_distro}, the limiting distribution of the total number of fixed points up to $\lfloor n/2\rfloor $ is counted by $Y_1$ and the total after is counted by $Y_2$.  Both are independent copies of the same random variable $Y$ with distribution
\begin{align*}
\prob(Y=k) &= \sum_{t=0}^\infty \prob( N=t ) { t \choose k} 2^{-t}  \\
&= \sum_{t=0}^\infty 2^{-2t-1}{t \choose k}\\
&= 2 *3^{-k-1}
\end{align*}
which is the probability mass function for a Geometric$(2/3)$ random variable.  
\end{proof}

\bibliographystyle{plain}
\bibliography{pattern}

\end{document}